\documentclass[11pt]{amsart}

\usepackage[margin=1.12in]{geometry}
\usepackage{amsmath,amssymb,amsthm,mathtools,mathrsfs}
\usepackage{enumitem}
\usepackage{booktabs}
\usepackage{tikz-cd}
\usepackage{xcolor}
\usepackage{hyperref}
\usepackage{orcidlink}

\definecolor{darkgoldenrod}{rgb}{0.72,0.53,0.04}
\definecolor{goldmetallic}{rgb}{0.83,0.69,0.22}
\hypersetup{
 colorlinks=true,
 linkcolor=darkgoldenrod,
 filecolor=brown,
 urlcolor=goldmetallic,
 citecolor=darkgoldenrod,
}

\numberwithin{equation}{section}

\newtheorem{theorem}{Theorem}[section]
\newtheorem{proposition}[theorem]{Proposition}
\newtheorem{lemma}[theorem]{Lemma}
\newtheorem{assumption}[theorem]{Assumption}

\newtheorem{definition}[theorem]{Definition}

\newcommand{\op}[1]{\operatorname{#1}}

\newcommand{\F}{\mathbb F}

\newcommand{\Z}{\mathbb Z}

\newcommand{\cO}{\mathcal O}

\title[Greenberg's \(\mu=0\) conjecture for lisse sheaves]
{Greenberg's \(\mu=0\) conjecture for lisse sheaves over global function fields}

\author[A.~Ray]{Anwesh Ray\, \orcidlink{0000-0001-6946-1559}}
\address{Chennai Mathematical Institute, Chennai, India}
\email{anwesh@cmi.ac.in}

\begin{document}

\begin{abstract}
Let $K$ be a global function field of characteristic \(p>0\) and $\ell\neq p$ be a prime number. We study Selmer groups
over a $\mathbb{Z}_\ell$-extension $K_\infty/K$. For a lisse \(\mathbb Z_\ell\)-sheaf we prove that the Pontryagin dual of the associated Selmer group is a finitely
generated torsion module over the Iwasawa algebra and has \(\mu\)-invariant
equal to zero.  This gives a positive-characteristic,
prime to $p$, analogue of Greenberg's \(\mu=0\) conjecture. Our result applies in particular to abelian varieties,
fine Selmer groups, and adjoint representations. We also prove an analogue of the
weak Leopoldt conjecture in this context over \(K_\infty\), and deduce that the framed
deformation ring of a residual representation is a
formal power series ring. The same conclusion holds for
the unframed deformation ring if the residual representation has no non-scalar endomorphisms.
\end{abstract}

\maketitle

\section{Introduction}
Let \(E/\mathbb Q\) be an elliptic curve and let \(\ell\) be an odd prime.
In the classical cyclotomic Iwasawa theory of elliptic curves, one studies the
\(\ell\)-primary Selmer group $\operatorname{Sel}_{\ell^\infty}(E/\mathbb Q_\infty)$
over the cyclotomic \(\mathbb Z_\ell\)-extension
\(\mathbb Q_\infty/\mathbb Q\). Its Pontryagin dual is a module over the
Iwasawa algebra $\Lambda=\mathbb Z_\ell[[\operatorname{Gal}(\mathbb Q_\infty/\mathbb Q)]]
        \simeq \mathbb Z_\ell[[T]]$. When this dual Selmer group is finitely generated and torsion over
\(\Lambda\), its structure is measured by the Iwasawa invariants \(\lambda\)
and \(\mu\).  Greenberg's \(\mu=0\) conjecture predicts, under the usual
ordinary hypotheses and an irreducibility hypothesis on the residual
representation, that the \(\mu\)-invariant vanishes; see
\cite[Conjecture 1.11]{GreenbergITEC}.  The irreducibility hypothesis is
essential in the number-field setting. When the residual representation is
reducible, positive \(\mu\)-invariants can occur, as in examples going back to
Mazur \cite{Mazurmain}.

\par In this paper, we prove a positive-characteristic analogue of
this \(\mu=0\) phenomenon over global function fields.  Let \(K\) be a global
function field of characteristic \(p>0\), with field of constants
\(k=\mathbb F_q\), and let \(\ell\neq p\).  Write $K=k(\mathscr X)$,
where \(\mathscr X/k\) is a smooth, projective, geometrically integral curve.
Let $k_\infty=\bigcup_{n\geq 0}\mathbb F_{q^{\ell^n}}$
be the unique \(\mathbb Z_\ell\)-extension of \(k\), and set $K_\infty=Kk_\infty$.
Then \(K_\infty/K\) is the only $\Z_\ell$-extension of $K$.

\par Put $\Gamma=\operatorname{Gal}(K_\infty/K)$ and $\Lambda=\mathbb Z_\ell[[\Gamma]]$ be the Iwasawa algebra. After choosing a topological generator \(\gamma\in\Gamma\), we identify
\(\Lambda\) with \(\mathbb Z_\ell[[T]]\) by sending \(\gamma\) to \(1+T\). The natural level of generality for the main theorem is that of lisse
\(\mathbb Z_\ell\)-sheaves on open curves.  Let $\mathscr U\subseteq \mathscr X$ be a nonempty affine open subscheme, and let \(\mathscr T\) be a lisse
\(\mathbb Z_\ell\)-sheaf on \(\mathscr U\), with geometric generic stalk \(T\)
finitely generated and free over \(\mathbb Z_\ell\).  Set $\mathscr W
        =
        \mathscr T\otimes_{\mathbb Z_\ell}\mathbb Q_\ell/\mathbb Z_\ell$ and consider Selmer groups $\operatorname{Sel}_{\mathcal F}(\mathscr W/K_\infty)$ defined using global cohomology over $G_S(K_\infty)=\operatorname{Gal}(K_S/K_\infty)$ where $S:=\mathscr X\setminus \mathscr U$, and local conditions at the places above \(S\). \par The local conditions are
assumed to be cartesian with respect to multiplication by \(\ell\) (in the sense of Assumption \ref{assumption:cartesian-local-conditions}). This hypothesis is
satisfied by the usual Kummer local conditions for abelian varieties, and it is
also the correct compatibility condition for the fine and unramified Selmer
structures considered below.

\begin{theorem}[Theorem \ref{thm:lisse-selmer-mu-zero}]
With respect to notation above, the dual Selmer group is a finitely generated torsion \(\Lambda\)-module whose $\mu$-invariant vanishes.
\end{theorem}
\par We also prove an analogue of the weak Leopoldt conjecture.  In the
classical number-field setting, let \(F\) be a number field, let \(p\) be a
prime, and let \(F_\infty/F\) be the cyclotomic $\Z_p$-extension. Let \(S\) be a finite set of places of
\(F\) containing the primes above \(p\), the archimedean places, and the primes
ramified in \(F_\infty/F\).  If \(V\) is a \(p\)-adic representation of
\(G_S(F)\), \(T\subset V\) is a \(G_S(F)\)-stable lattice, and $W=V/T$
then the weak Leopoldt conjecture predicts the vanishing $H^2(G_S(F_\infty),W)$. For further details, we refer to \cite[Appendix B]{perrinriou}.
\par In the present paper we prove the corresponding statement for
lisse \(\mathbb Z_\ell\)-sheaves.
\begin{theorem}[Theorem \ref{thm:weak-leopoldt-constant}]
Let \(\mathscr T\) be a lisse \(\mathbb Z_\ell\)-sheaf on
\(\mathscr U=\mathscr X\setminus S\), with \(\mathscr U\) affine, and put
\[
        \mathscr W
        =
        \mathscr T\otimes_{\mathbb Z_\ell}\mathbb Q_\ell/\mathbb Z_\ell.
\]
Then the weak Leopoldt conjecture holds for $\mathscr{W}$, i.e.,
\[
        H^2(G_S(K_\infty),\mathscr W)=0.
\]
\end{theorem}

\par We also obtain an application to deformation rings.  This should be
compared with the number-field theory of ordinary deformations in cyclotomic
towers.  Hida initiated the study of ordinary deformation rings in this setting;
in particular, he studied the variation of ordinary Galois deformation rings
over the cyclotomic \(\mathbb Z_p\)-extension of a totally real field
\cite{HidaHilbert}. More recently, Burungale and Clozel proved
that ordinary deformations are unobstructed in the cyclotomic limit under
natural hypotheses, including a noetherian hypothesis for the ordinary deformation
ring and the vanishing of certain adjoint \(\mu\)-invariants
\cite{BurungaleClozel}.
\par Let $\bar r:G_S(K)\longrightarrow \operatorname{GL}_d(\mathbb F)$
be a residual representation, where \(\mathbb F\) is a finite field of
characteristic \(\ell\), and let \(\mathcal O=W(\mathbb F)\).  Restricting to
\(G_S(K_\infty)\), write $\bar r_\infty=\bar r|_{G_S(K_\infty)}$.
The adjoint representation
\[
        \operatorname{Ad}\bar r
        =
        \operatorname{End}_{\mathbb F}(V_{\bar r})
\]
defines a constructible \(\mathbb F\)-sheaf on \(\mathscr U\). We prove that
\[
        H^2(G_S(K_\infty),\operatorname{Ad}\bar r)=0
\]
and that $H^1(G_S(K_\infty),\operatorname{Ad}\bar r)$ is finite dimensional.
\begin{theorem}[Theorem \ref{thm:constant-tower-deformation-unobstructed}]
The framed deformation functor of \(\bar r_\infty\) is formally
smooth over \(\mathcal O=W(\mathbb F)\), and its representing ring is a formal
power series ring
\[
        R_{\bar r_\infty}^{\square}
        \simeq
        \mathcal O[[x_1,\ldots,x_m]]
\]
in finitely many variables.  If moreover
\[
        \operatorname{End}_{\mathbb F[G_S(K_\infty)]}(V_{\bar r})=\mathbb F,
\]
then the unframed deformation functor of \(\bar r_\infty\) is representable and
formally smooth over \(\mathcal O\).  In particular,
\[
        R_{\bar r_\infty}
        \simeq
        \mathcal O[[y_1,\ldots,y_{m'}]]
\]
for some \(m'\geq 0\).
\end{theorem}
\noindent We note that in our setting, there is no noetherian hypothesis on the deformation rings and we do not impose a $\mu=0$ hypothesis for the adjoint representation. 
\par Iwasawa theory over global function fields has been developed in several
directions. The work of Bandini and Longhi on Selmer groups for elliptic
curves over function fields
\cite{bandinilonghi} is especially close in spirit to the present paper. They
establish foundational control and cotorsion results for elliptic curves over
function fields in the prime-to-characteristic setting.  Bandini and Valentino
study Selmer groups of abelian varieties over \(\ell\)-adic Lie extensions of
global function fields in \cite{BandiniValentino}.  Bandini, Bars and
Longhi study characteristic ideals and their relation to Selmer
groups in \cite{BandiniBarsLonghi}. These works establish control theorems, cotorsion results, and structural
properties of Selmer groups over Iwasawa algebras in the function-field
setting.
\par Witte \cite{witte} established a noncommutative Iwasawa main conjecture for general $\ell$-adic representations over global function fields. Deng, Kezuka, Li, and Lim \cite{DengKezukaLiLim} proved the $\mathfrak M_H(G)$-conjecture for the $\ell$-primary Selmer groups of abelian varieties and, in particular, the vanishing of the cyclotomic $\mu$-invariant. Thus, in the case of abelian varieties, the corresponding $\mu$-vanishing result in the present paper is already contained in their work. Our aim here is to develop the sheaf-theoretic framework and use it as a basis for the further results proved below. In particular, our results apply to Greenberg Selmer groups associated with ordinary lisse sheaves, including those arising from symmetric powers of Tate modules of abelian varieties and from trace-zero adjoint sheaves.

\section{Selmer groups of lisse etale sheaves}
\subsection{Lisse $\Z_\ell$-sheaves}\label{s 2.1}
\par We introduce basic notation and the fundamental objects of study in this article. Let $p$ be a prime number, let $q$ be a power of $p$, and set
$k:=\mathbb F_q$. Let $K$ be a global function field of characteristic $p$
over $k$. In other words, $K$ is a finite separable algebraic extension of the function field $k(X)$. Choose separable closures $\bar{k}/k$ and
$\overline{K}/K$ such that $\bar{k}\subset \overline{K}$. We write $G_k:=\operatorname{Gal}(\bar{k}/k)$ and $G_K:=\operatorname{Gal}(\overline{K}/K)$. We assume without loss of generality that $k$ is the field of constants of $K$, i.e.,
\begin{equation}\label{fieldofconstants}
K\cap \bar{k}=k.
\end{equation}
\par Let $\sigma_k\in G_k$ denote the arithmetic Frobenius, given by $\sigma_k(x)=x^q$. Note that $G_k$ is topologically generated by
$\sigma_k$. Thus there is a canonical identification $G_k\simeq \widehat{\mathbb Z}$, sending $\sigma_k$ to $1$. By the linear disjointness condition \eqref{fieldofconstants}, restriction induces a natural isomorphism
\[
\operatorname{Gal}(K\cdot\bar{k}/K)\xrightarrow{\sim}G_k\simeq \widehat{\mathbb Z}.
\]
We shall use the decomposition
$\widehat{\mathbb Z}\simeq \prod_r \mathbb Z_r$,
where $r$ ranges over all prime numbers. Let $\ell\neq p$ be a prime number. We denote by $k_\infty\subset \bar{k}$
the unique $\mathbb Z_\ell$-extension of $k$; explicitly,
\[
k_\infty=\bigcup_{n\geq 0}\mathbb F_{q^{\ell^n}}.
\]
Then $\operatorname{Gal}(k_\infty/k)\simeq \mathbb Z_\ell$, and, under the above identification $G_k\simeq \prod_r\mathbb Z_r$, one has
\[
\operatorname{Gal}(\bar{k}/k_\infty)\simeq \prod_{r\neq \ell}\mathbb Z_r.
\]

Let $\mathscr X/k$ be the smooth, projective, geometrically integral curve
with function field $K=k(\mathscr X)$. The condition
\eqref{fieldofconstants} translates to the assumption that $k$ is the
field of constants of $\mathscr X$. A $\Z_\ell$-extension $K_\infty/K$ is an infinite Galois extension for which $\Gamma:=\op{Gal}(K_\infty/K)$ is isomorphic to $\Z_\ell$. The \emph{constant} $\mathbb Z_\ell$-extension of $K$ is defined by $K_\infty^c:=K\cdot k_\infty$.
Since $K$ and $\bar{k}$ are linearly disjoint over $k$, restriction induces isomorphisms
\begin{equation}\label{Gammaiso}
\Gamma\simeq \operatorname{Gal}(k_\infty/k)\simeq \mathbb Z_\ell.
\end{equation}

\begin{proposition}\label{prop:all-zl-extensions-constant}
Every
\(\mathbb Z_\ell\)-extension of \(K\) is the constant
\(\mathbb Z_\ell\)-extension.
\end{proposition}

\begin{proof}
Let \(K_\infty/K\) be a \(\mathbb Z_\ell\)-extension, and write
\[
        \Gamma=\operatorname{Gal}(K_\infty/K)\simeq \mathbb Z_\ell.
\]
We first show that \(K_\infty/K\) is unramified at every place of \(K\).

Let \(v\) be a place of \(K\), and choose a place \(w\) of \(K_\infty\) above
\(v\).  The inertia subgroup \(I_w\subseteq \Gamma\) is a closed subgroup of
\(\mathbb Z_\ell\).  Hence \(I_w\) is either trivial or infinite.  On the
other hand, local class field theory shows that the ramified pro-\(\ell\)
quotient of \(G_{K_v}^{\mathrm{ab}}\) is finite.  Indeed,
\[
        K_v^\times\simeq \langle \varpi_v\rangle\times \mathcal O_v^\times,
        \quad\text{and}\quad
        \mathcal O_v^\times\simeq k_v^\times\times U_v^1,
\]
where \(U_v^1\) is a pro-\(p\) group and \(k_v^\times\) is finite.  Since
\(\ell\neq p\), the pro-\(\ell\) part of \(\mathcal O_v^\times\) is finite.
The only infinite pro-\(\ell\) quotient of \(K_v^\times\) comes from the
uniformizer factor, and corresponds to the unramified local
\(\mathbb Z_\ell\)-extension.  Therefore the image of local inertia in
\(\Gamma\) is finite.  Since \(\mathbb Z_\ell\) has no nontrivial finite
subgroups, this image is trivial.  Thus \(K_\infty/K\) is unramified at \(v\).
Since \(v\) was arbitrary, \(K_\infty/K\) is everywhere unramified.

Let \(X/k\) be the smooth projective geometrically integral curve with
function field \(K\).  By global class field theory for function fields (cf.~\cite{Rosen}), the
Galois group of the maximal everywhere unramified abelian extension of \(K\)
fits into an exact sequence
\[
        0
        \longrightarrow
        \operatorname{Pic}^0(X)(k)
        \longrightarrow
        \operatorname{Gal}(K^{\mathrm{ur,ab}}/K)
        \xrightarrow{\deg}
        \widehat{\mathbb Z}
        \longrightarrow
        0.
\]
Here the degree map corresponds to extension of constants, and
\(\operatorname{Pic}^0(X)(k)\) is finite.  Passing to maximal pro-\(\ell\)
quotients gives
\[
        0
        \longrightarrow
        \operatorname{Pic}^0(X)(k)[\ell^\infty]
        \longrightarrow
        \operatorname{Gal}(K^{\mathrm{ur,ab}}/K)^{(\ell)}
        \longrightarrow
        \mathbb Z_\ell
        \longrightarrow
        0,
\]
where the kernel is finite. Consider the quotient map
\[\operatorname{Gal}(K^{\mathrm{ur,ab}}/K)^{(\ell)}\twoheadrightarrow \op{Gal}(K_\infty/K).\]
The finite subgroup
\(\operatorname{Pic}^0(X)(k)[\ell^\infty]\) maps trivially to this quotient.
Hence the quotient coincides with the degree map
\[
        \operatorname{Gal}(K^{\mathrm{ur,ab}}/K)^{(\ell)}
        \longrightarrow
        \mathbb Z_\ell.
\]
But this degree quotient corresponds precisely the constant
\(\mathbb Z_\ell\)-extension \(K^c_\infty/K\).  Therefore
\[
        K_\infty=K^c_\infty.
\]
\end{proof}
When \(\ell=p\), the situation is completely
different. In this case $K$ has infinitely many independent $\Z_p$-extensions. This is to say that there are infinitely many $\Z_p$-extensions $K_\infty^i$ such that $K_\infty^{i+1}$ is not contained in the composition of $K_\infty^j$ for $j=1, \dots, i$. These extensions are naturally described by
Artin--Schreier--Witt theory rather than by the tame prime-to-\(p\) class field
theory used above (see Kosters--Wan
\cite{kosters2016arithmetic, kosterswan2}).

\par We shall henceforth denote the constant $\Z_\ell$-extension by $K_\infty/K$. In light of the above result, this is the only $\Z_\ell$-extension of $K$. For every integer $n\geq 0$, let $K_n/K$ be the subextension for which $[K_n:K]=\ell^n$. The field $K_n$ is called the \emph{$n$-th layer} and we have a tower of function fields
\[K=K_0\subseteq K_1\subseteq \dots \subseteq K_n\subseteq K_{n+1}\subseteq \dots \subseteq \bigcup_{m\geq 0}K_m=K_\infty. \]
The Galois group $\op{Gal}(K_n/K)$ is identified with the quotient $\Gamma/\Gamma^{\ell^n}$. The \emph{Iwasawa algebra} associated to $\Gamma$ is defined to be the completed group algebra \[\Lambda:=\varprojlim_n \Z_\ell\left[\Gamma/\Gamma^{\ell^n}\right].\]
Choose a topological generator \(\gamma\in\Gamma\) which maps to to $1\in \Z_\ell$ under the isomorphism \eqref{Gammaiso}. We identify
\(\Lambda\) with the formal power series ring \(\mathbb Z_\ell[[T]]\) upon identifying \(\gamma\) with the element \(1+T\).
\par We now introduce the sheaves which shall give rise to Galois representations. Let $\mathscr U\subseteq \mathscr X$ be a nonempty open affine subscheme. Since
$\mathscr X$ is a smooth projective geometrically integral curve over $k$,
the complement $\mathscr X\setminus \mathscr U$ is a finite set of closed points $S$. The points in $S$ give rise to places $w$ of the function field $K$. We shall identify the set $S$ with a set of places $w$. Thus $\mathscr U$ should be thought of as the curve $\mathscr X$
with finitely many places removed. Each of these places may be assumed, after replacing $k$ by a finite extension, to be defined over $k$. Here
$\ell\neq p$, so $\ell$ is invertible on $\mathscr U$.
\begin{definition}By a \emph{lisse
$\mathbb Z_\ell$-sheaf} we mean an inverse system $\mathscr T=\{\mathscr T_n\}_{n\geq 1}$ where each $\mathscr T_n$ is a finite type, locally constant, constructible sheaf of $\mathbb Z/\ell^n\mathbb Z$-modules on $\mathscr U_{\mathrm{\acute et}}$ and
the transition maps satisfy
\[
\mathscr T_{n+1}\otimes_{\mathbb Z/\ell^{n+1}\mathbb Z}
\mathbb Z/\ell^n\mathbb Z
\simeq \mathscr T_n.
\]
\end{definition}
\par Choose a geometric generic point \(\bar\eta\) of \(\mathscr U\). For each \(n\), the stalk of \(\mathscr T_n\) at the geometric generic point
\(\bar\eta\) is defined by
\[
        (\mathscr T_n)_{\bar\eta}
        :=
        \varinjlim_{(\bar\eta\to V\to \mathscr U)}
        \mathscr T_n(V),
\]
where the limit is taken over all étale neighbourhoods of \(\bar\eta\).  Since
\(\mathscr T_n\) is locally constant constructible, this stalk is a finite
\(\mathbb Z/\ell^n\mathbb Z\)-module equipped with a continuous action of
\(\pi_1^{\mathrm{\acute et}}(\mathscr U,\bar\eta)\).  For the inverse system
\(\mathscr T=\{\mathscr T_n\}_{n\geq 1}\), we set
\[
        T:=\varprojlim_n(\mathscr T_n)_{\bar\eta}.
\]
Then \(T\) is a \(\mathbb Z_\ell\)-module with a continuous action of
\(\pi_1^{\mathrm{\acute et}}(\mathscr U,\bar\eta)\).  We assume throughout
that \(T\) is finite free over \(\mathbb Z_\ell\).
By abuse of notation, we shall also set
\[\mathscr T:=\varprojlim_n \mathscr T_n\] and we identify $\mathscr T_n$ with $\mathscr T/\ell^n \mathscr T$.
\par A lisse $\Z_\ell$-sheaf $\mathscr T$ gives rise to a continuous representation
\[
\rho_{\mathscr T}:\pi_1^{\mathrm{et}}(\mathscr U,\bar\eta)
\longrightarrow \operatorname{Aut}_{\mathbb Z_\ell}(T)
\simeq \operatorname{GL}_r(\mathbb Z_\ell),
\]
where $r=\operatorname{rank}_{\mathbb Z_\ell}T$; see \cite[Section 59.65, Lemma 59.65.1]{Stacks}. The continuity condition is imposed with respect to the profinite topology
on $\pi_1^{\mathrm{et}}(\mathscr U,\bar\eta)$ and the $\ell$-adic topology
on $\operatorname{GL}_r(\mathbb Z_\ell)$. Thus $H^1\left(\pi_1^{\mathrm{et}}(\mathscr U,\bar\eta), T\right)$ shall refer to continuous cohomology classes on $\pi_1^{\mathrm{et}}(\mathscr U,\bar\eta)$ with values in $T$. On the other hand, $H^1_{\mathrm{et}}(\mathscr{U}, \mathscr T)$ shall be used for the space of \'etale $1$-cohomology classes.
\par The étale
fundamental group of $\mathscr U$ fits into the standard exact sequence
\[
1\longrightarrow \pi_1^{\mathrm{et}}(\mathscr U_{\bar{k}},\bar\eta)
\longrightarrow \pi_1^{\mathrm{et}}(\mathscr U,\bar\eta)
\longrightarrow G_k
\longrightarrow 1.
\]
Thus the representation attached to $\mathscr T$ contains both geometric
information, coming from the action of
$\pi_1^{\mathrm{et}}(\mathscr U_{\bar{k}},\bar\eta)$, and arithmetic
information, coming from the action of the Frobenius element in $G_k$.
If \(S=\mathscr X\setminus\mathscr U\), then
\[
        \pi_1^{\mathrm{et}}(\mathscr U,\bar\eta)
        \simeq
        \operatorname{Gal}(K_S/K)
\]where $K_S\subset \overline{K}$ is the maximal separable extension of $K$ unramified ouside $S$. Thus $T$
may also be viewed as a continuous \(\mathbb Z_\ell\)-representation of
\(G_K\), unramified at every closed point of \(\mathscr U\).

\par Associated to $\mathscr T$, we define the associated discrete $\ell$-primary torsion sheaf by $\mathscr W
        :=
        \varinjlim_n \mathscr T_n$, where the transition maps
\[
        \mathscr T_n\longrightarrow \mathscr T_{n+1}
\]
are induced by multiplication by $\ell$. We note that $\mathscr W$ is the sheaf associated to
the discrete $\pi_1^{\mathrm{et}}(\mathscr U,\bar\eta)$-module $W:=T\otimes_{\mathbb Z_\ell}\mathbb Q_\ell/\mathbb Z_\ell$. We also set $\mathscr M:=\mathscr T/\ell\mathscr T=\mathscr T_1$. These sheaves are related by a short exact sequence
\[0\rightarrow \mathscr M\rightarrow \mathscr W\xrightarrow{\times \ell }\mathscr W\rightarrow 0. \]
\subsection{Selmer groups}
\par We now define the notion of a Selmer structures on \(\mathscr W\) and \(\mathscr M\). Let \(F/K\) be a separable extension contained in $K_S$. We
write
\[
        G_S(F):=\operatorname{Gal}(K_S/F).
\]
We shall use the shorthand 
\[H^i(K_S/F, \cdot):=H^i(G_S(F), \cdot).\]
Let \(\mathscr U_F\) be the open curve with function field \(F\) defined by $\mathscr U_F:=\mathscr X_F\setminus S$. The global cohomology groups
appearing in the definition of the Selmer group are
\[
        H^i_{\mathrm{et}}(\mathscr U_F,\mathscr M)
        =
        H^i(K_S/F,\mathscr M),
        \quad\text{and}\quad
        H^i_{\mathrm{et}}(\mathscr U_F,\mathscr W)
        =
        H^i(K_S/F,\mathscr W).
\]

\par For a place \(w\) of \(F\), let \(F_w\) be the completion of \(F\) at \(w\),
and let \(G_{F_w}\) be its absolute Galois group. Denote by $H^1(F_w,\mathscr M)$ and $H^1(F_w,\mathscr W)$ the local Galois cohomology groups $H^1(G_{F_w},\mathscr M)$ and $H^1(G_{F_w},\mathscr W)$ respectively. Let \(S_F\) be the set of places of \(F\) above
\(S\). A Selmer structure \(\mathcal F\) associated to $(\mathscr{W}, \mathscr{M})$ consists of local subgroups
\[
        \mathcal L_w(\mathscr W)
        \subseteq
        H^1(F_w,\mathscr W)\quad \text{and} \quad\mathcal L_w(\mathscr M)
        \subseteq
        H^1(F_w,\mathscr M)
\]
for each \(w\in S_F\).
The Selmer group associated to this datum is
\[
\operatorname{Sel}_{\mathcal F}(\mathscr W/F)
=
\ker\left\{
H^1(G_S(F),\mathscr W)
\longrightarrow
\prod_{w\in S_F}
\frac{H^1(G_{F_w},\mathscr W)}{\mathcal L_w(\mathscr W)}
\right\},
\]
and the residual Selmer group is
\[
\operatorname{Sel}_{\mathcal F}(\mathscr M/F)
=
\ker\left\{
H^1(G_S(F),\mathscr M)
\longrightarrow
\prod_{w\in S_F}
\frac{H^1(G_{F_w},\mathscr M)}{\mathcal L_w(\mathscr M)}
\right\}.
\]
\begin{assumption}\label{assumption:cartesian-local-conditions}
We assume that the Selmer structure is cartesian with respect to
multiplication by \(\ell\).  More precisely, for every place \(w\mid S\), let
\[
        \beta_w:
        H^1(F_w,\mathscr M)
        \longrightarrow
        H^1(F_w,\mathscr W)[\ell]
\]
be the natural map induced by the exact sequence
\[
        0
        \longrightarrow
        \mathscr M
        \longrightarrow
        \mathscr W
        \xrightarrow{\ell}
        \mathscr W
        \longrightarrow
        0.
\]
We require that
\[
        \mathcal L_w(\mathscr M)
        =
        \beta_w^{-1}\bigl(\mathcal L_w(\mathscr W)[\ell]\bigr).
\]
\end{assumption}
\begin{lemma}\label{lem:lisse-selmer-comparison}
The natural map
\[
        \operatorname{Sel}_{\mathcal F}(\mathscr M/F)
        \longrightarrow
        \operatorname{Sel}_{\mathcal F}(\mathscr W/F)[\ell]
\]
is surjective.  In particular, if
\(\operatorname{Sel}_{\mathcal F}(\mathscr M/F)\) is finite, then
\(\operatorname{Sel}_{\mathcal F}(\mathscr W/F)[\ell]\) is finite.
\end{lemma}

\begin{proof}
Let
\[
        \beta_F:
        H^1_{\mathrm{et}}(\mathscr U_F,\mathscr M)
        \longrightarrow
        H^1_{\mathrm{et}}(\mathscr U_F,\mathscr W)
\]
be the map induced by the short exact sequence
\[
        0
        \longrightarrow
        \mathscr M
        \longrightarrow
        \mathscr W
        \xrightarrow{\ell}
        \mathscr W
        \longrightarrow
        0.
\]
The associated long exact sequence gives
\[
        \operatorname{im}(\beta_F)
        =
        H^1_{\mathrm{et}}(\mathscr U_F,\mathscr W)[\ell].
\]
Thus every class
\[
        c\in
        \operatorname{Sel}_{\mathcal F}(\mathscr W/F)[\ell]
\]
admits a global lift
\[
        \widetilde c\in
        H^1_{\mathrm{et}}(\mathscr U_F,\mathscr M)
\]
such that
\[
        \beta_F(\widetilde c)=c.
\]

We show that this lift automatically satisfies the residual Selmer conditions.
Let \(w\in S_F\).  By functoriality of localization, the diagram
\[
\begin{tikzcd}
H^1_{\mathrm{et}}(\mathscr U_F,\mathscr M)
    \arrow[r,"\beta_F"]
    \arrow[d,"\operatorname{loc}_w"]
&
H^1_{\mathrm{et}}(\mathscr U_F,\mathscr W)[\ell]
    \arrow[d,"\operatorname{loc}_w"]
\\
H^1(F_w,\mathscr M)
    \arrow[r,"\beta_w"]
&
H^1(F_w,\mathscr W)[\ell]
\end{tikzcd}
\]
commutes.  Since \(c\) belongs to
\(\operatorname{Sel}_{\mathcal F}(\mathscr W/F)\), we have
\[
        \operatorname{loc}_w(c)\in \mathcal L_w(\mathscr W)[\ell].
\]
Therefore
\[
        \beta_w\bigl(\operatorname{loc}_w(\widetilde c)\bigr)
        =
        \operatorname{loc}_w(c)
        \in
        \mathcal L_w(\mathscr W)[\ell].
\]
By the cartesian hypothesis,
\[
        \operatorname{loc}_w(\widetilde c)
        \in
        \beta_w^{-1}\bigl(\mathcal L_w(\mathscr W)[\ell]\bigr)
        =
        \mathcal L_w(\mathscr M).
\]
This holds for every \(w\in S_F\).  Hence
\[
        \widetilde c\in
        \operatorname{Sel}_{\mathcal F}(\mathscr M/F).
\]
Thus every class in
\(\operatorname{Sel}_{\mathcal F}(\mathscr W/F)[\ell]\) is the image of a
class in \(\operatorname{Sel}_{\mathcal F}(\mathscr M/F)\), proving the
surjectivity.

The final assertion follows immediately, since a quotient of a finite group is
finite.
\end{proof}

\subsection{Abelian varieties}
\par We spell out the case of abelian varieties, since it is the motivating
example for the general definition. Let \(A/K\) be an abelian variety, and let
\(S\) be a finite set of geometric points of \(\mathscr X\) containing all places
where \(A\) has bad reduction. Assume without loss of generality (upon extending the field $k$) that all points of $S$ are defined over $k$. Setting \(\mathscr U=\mathscr X\setminus S\), assume that $\mathscr U$ is affine. We find that
\(A\) extends to an abelian scheme $\mathcal A$ over $\mathscr U$, i.e., 
\[\mathscr A\times_{\mathscr{U}} \op{Spec} K\simeq A.\]
Since \(\ell\neq p\), the finite group schemes \(\mathcal A[\ell^n]\) are
finite etale over \(\mathscr U\). The $\ell$-adic Tate-module $T_\ell(\mathcal A)=\varprojlim_n \mathcal A[\ell^n]$ is a lisse \(\mathbb Z_\ell\)-sheaf on \(\mathscr U\). In accordance with the notation from Section \ref{s 2.1},
we set \[\begin{split}&\mathscr T=T_\ell(\mathcal A),\\
&\mathscr W=\mathscr T\otimes_{\mathbb Z_\ell}\mathbb Q_\ell/\mathbb Z_\ell
        =A[\ell^\infty], \quad \text{and}\\&\mathscr M=\mathscr T/\ell\mathscr T=A[\ell].\\
        \end{split}\]
\par Let \(F/K\) be a finite separable algebraic extension and let \(w\) be a place of \(F\).
The usual Kummer sequence
\[
        0\longrightarrow A[\ell^n]
        \longrightarrow A
        \xrightarrow{\ell^n}
        A
        \longrightarrow 0
\]
gives a local connecting homomorphism
\[
        \kappa_{w}^{(n)}:A(F_w)/\ell^n A(F_w)
        \longrightarrow
        H^1(F_w,A[\ell^n]).
\]
Passing to the direct limit over \(n\), one obtains the local Kummer map
\[
        \kappa_w:
        A(F_w)\otimes_{\mathbb Z_\ell}\mathbb Q_\ell/\mathbb Z_\ell
        \longrightarrow
        H^1(F_w,A[\ell^\infty]).
\]
The usual local condition for the \(\ell^\infty\)-Selmer group is the image
\[
        \mathcal L_w(A[\ell^\infty])
        :=
        \operatorname{im}(\kappa_w)
        \subseteq
        H^1(F_w,A[\ell^\infty]).
\]
We shall set
\[
        \overline{\kappa}_{w}:=\kappa_w^{(1)}:
        A(F_w)/\ell A(F_w)
        \longrightarrow
        H^1(F_w,A[\ell]),
\]
and the residual local condition is
\[
        \mathcal L_w(A[\ell])
        :=
        \op{image}\left(\overline{\kappa}_w\right)
        \subseteq
        H^1(F_w,A[\ell]).
\]
\par With these local conditions, the global Selmer group over $F$ is defined as follows:
\[
\operatorname{Sel}_{\ell^\infty}(A/F)
=
\ker\left\{
H^1(G_S(F),A[\ell^\infty])
\longrightarrow
\prod_{w\in S_F}
\frac{H^1(F_w,A[\ell^\infty])}{\mathcal L_w(A[\ell^\infty])}
\right\}.
\]
Here \(S_F\) denotes the set of places of \(F\) above \(S\), and
\(G_S(F)\) is the Galois group of the maximal extension of \(F\) unramified
outside \(S_F\). We set
\[\op{Sel}_{\ell^\infty}(A/K_\infty):=\varinjlim_n \op{Sel}_{\ell^\infty}(A/K_n).\]
\begin{lemma}\label{lem:kummer-cartesian-local-condition}
Let \(F/K\) be a finite separable extension, let \(A/F\) be an abelian
variety, and let \(w\) be a place of \(F\). Put \(L=F_w\), and assume that
\(\ell\neq \operatorname{char}(F)\). Let
\[
        \beta_w:
        H^1(L,A[\ell])
        \longrightarrow
        H^1(L,A[\ell^\infty])[\ell]
\]
be the map induced by the inclusion \(A[\ell]\subseteq A[\ell^\infty]\).
Then the Kummer local conditions are cartesian with respect to multiplication
by \(\ell\), namely
\[
        \mathcal L_w(A[\ell])
        =
        \beta_w^{-1}
        \left(\mathcal L_w(A[\ell^\infty])\right).
\]
\end{lemma}

\begin{proof}
The
sequence
\[
        0
        \longrightarrow
        A[\ell]
        \longrightarrow
        A
        \xrightarrow{\ell}
        A
        \longrightarrow
        0
\]
gives an exact sequence
\[
        0
        \longrightarrow
        A(L)/\ell A(L)
        \longrightarrow
        H^1(L,A[\ell])
        \xrightarrow{q_\ell}
        H^1(L,A)[\ell]
        \longrightarrow
        0.
\]
Thus $\mathcal L_w(A[\ell])
        =
        \ker(q_\ell)$. Taking the direct limit over the Kummer sequences for
\(A[\ell^n]\) gives an exact sequence
\[
        0
        \longrightarrow
        A(L)\otimes_{\mathbb Z}\mathbb Q_\ell/\mathbb Z_\ell
        \longrightarrow
        H^1(L,A[\ell^\infty])
        \xrightarrow{q_\infty}
        H^1(L,A)[\ell^\infty]
        \longrightarrow
        0,
\]
and hence $\mathcal L_w(A[\ell^\infty])
        =
        \ker(q_\infty)$.

The inclusion \(A[\ell]\subseteq A[\ell^\infty]\) is compatible with the
Kummer sequences.  Therefore we have a commutative diagram
\[
\begin{tikzcd}[column sep=large,row sep=large]
H^1(L,A[\ell])
    \arrow[r,"q_\ell"]
    \arrow[d,"\beta_w"]
&
H^1(L,A)[\ell]
    \arrow[d,hook]
\\
H^1(L,A[\ell^\infty])
    \arrow[r,"q_\infty"]
&
H^1(L,A)[\ell^\infty].
\end{tikzcd}
\]
Let \(x\in H^1(L,A[\ell])\).  Then
\[
        \beta_w(x)\in \mathcal L_w(A[\ell^\infty])[\ell]
\]
if and only if
\[
        q_\infty(\beta_w(x))=0.
\]
By commutativity of the diagram, this is equivalent to the image of
\(q_\ell(x)\) in \(H^1(L,A)[\ell^\infty]\) being zero.  Since the natural map
\[
        H^1(L,A)[\ell]\hookrightarrow H^1(L,A)[\ell^\infty]
\]
is injective, this is equivalent to
\[
        q_\ell(x)=0.
\]
But \(\ker(q_\ell)=\mathcal L_w(A[\ell])\).  Hence
\[
        \beta_w^{-1}
        \left(\mathcal L_w(A[\ell^\infty])[\ell]\right)
        =
        \mathcal L_w(A[\ell]),
\]
as required.
\end{proof}
\par Let \[\mathscr T_A:=T_\ell(\mathcal A):=\varprojlim_n \mathcal{A}/\ell^n\mathcal{A}\] be the $\ell$-adic lisse \(\mathbb Z_\ell\)-sheaf on \(\mathscr U\). We note that 
\[T_\ell(\mathcal A)_{\bar{\eta}}\simeq T_\ell(A),\] where $\bar{\eta}$ is the generic point of $\mathscr U$ and $T_\ell(A)$ is the $\ell$-adic Tate-module of $A$. For an
integer \(m\geq 0\) and any integer \(r\), we may form the Tate twist of the symmetric power:
\[
        \operatorname{Sym}^m \mathscr T_A(r)
        :=
        \operatorname{Sym}^m \mathscr T_A\otimes_{\mathbb Z_\ell}\mathbb Z_\ell(r).
\]
This is again lisse, since both \(\operatorname{Sym}^m \mathscr T_A\) and
\(\mathbb Z_\ell(r)\) are lisse \(\mathbb Z_\ell\)-sheaves. The two most
natural choices of Selmer structure are the standard unramified local
condition and the fine local condition at all primes. The latter is particularly important
for the adjoint sheaf, since it is the local condition naturally related to the
obstruction groups appearing in deformation theory.

\subsection{Fine Selmer groups}

\par We now specialize the preceding formalism to the fine Selmer structure.
Let \(F/K\) be a separable algebraic extension, and let \(S_F\) denote the set
of places of \(F\) lying above \(S\).  Let \(\mathscr W\) be a discrete
\(\ell\)-primary sheaf arising from a lisse \(\mathbb Z_\ell\)-sheaf
\(\mathscr T\) on \(\mathscr U\).  The fine Selmer group of $\Omega\in \{\mathscr W, \mathscr M\}$ over
\(F\), with respect to \(S\), is
\[
        R_S(\Omega/F)
        :=
        \ker\left\{
        H^1(G_S(F),\Omega)
        \longrightarrow
        \prod_{w\in S_F} H^1(F_w,\Omega)
        \right\}.
\]
Thus a global class belongs to the fine Selmer group precisely when it is
locally trivial at every place above \(S\). The fine Selmer
structure is the Selmer structure for which
\[
        \mathcal L_w(\Omega)=0
        \qquad
        \text{for every } w\in S_F.
\]
For the comparison with
\(\ell\)-torsion in \(R_S(\mathscr W/F)\), the natural residual local condition
is the cartesian one.  Namely, let
\[
        \beta_w:
        H^1(F_w,\mathscr M)
        \longrightarrow
        H^1(F_w,\mathscr W)[\ell]
\]
be the map induced by
\[
        0
        \longrightarrow
        \mathscr M
        \longrightarrow
        \mathscr W
        \xrightarrow{\ell}
        \mathscr W
        \longrightarrow 0.
\]
For the fine Selmer structure on \(\mathscr W\), the residual fine
local condition is
\[
        \mathcal L_w^{\mathrm{cart}}(\mathscr M)
        :=
       \beta_w^{-1}\left(\mathcal L_w(\mathscr W)\right)=\op{ker}\beta_w.
\]
We therefore define the residual cartesian fine Selmer group by
\[
        R_S^{\mathrm{cart}}(\mathscr M/F)
        :=
        \ker\left\{
        H^1(G_S(F),\mathscr M)
        \longrightarrow
        \prod_{w\in S_F}
        H^1(F_w,\mathscr W)[\ell]\right\}.
\]
Clearly, the fine Selmer structure is compatible with
multiplication by \(\ell\) in the sense of Assumption
\ref{assumption:cartesian-local-conditions}.

\begin{lemma}\label{lem:fine-comparison}
Let \(F/K\) be a separable algebraic extension. Then there is a surjection
\[
R_S^{\mathrm{cart}}(\mathscr M/F)
\twoheadrightarrow
R_S(\mathscr W/F)[\ell].
\]
Further, suppose that all primes of $S$ are finitely decomposed in $F$. Then if \(R_S(\mathscr M/F)\) is finite, then
\(R_S(\mathscr W/F)[\ell]\) is finite.
\end{lemma}

\begin{proof}
This is Lemma \ref{lem:lisse-selmer-comparison} applied to the fine Selmer
structure. Thus if \(R_S^{\mathrm{cart}}(\mathscr M/F)\) is finite, then
\(R_S(\mathscr W/F)[\ell]\) is finite. For each $w\in S_F$, 
\[\mathcal L_w^{\mathrm{cart}}(\mathscr M)\simeq \frac{H^0(F_w, \mathscr{W})}{\ell H^0(F_w, \mathscr{W})}\] is finite. By assumption, the set $S_F$ is finite. Therefore if \(R_S(\mathscr M/F)\) is finite, then so is \(R_S^{\mathrm{cart}}(\mathscr M/F)\). It follows that \(R_S(\mathscr W/F)[\ell]\) is finite as well.
\end{proof}

We now record the behavior of fine Selmer groups under enlargement of
\(S\). 
\begin{lemma}\label{lem:fine-selmer-independent-of-S}
Let \(S'\supset S\) be a finite set of closed points of \(\mathscr X\).
Then the natural map identifies the fine Selmer groups
\[
        R_S(\mathscr W/K_\infty)
        =
        R_{S'}(\mathscr W/K_\infty).
\]
\end{lemma}

\begin{proof}
It is enough to treat the case \(S'=S\cup\{v\}\), with \(v\notin S\).  Let
\(w\) be a place of \(K_\infty\) above \(v\).  Since \(\mathscr W\) is
unramified at \(v\), any class in \(H^1(G_S(K_\infty),\mathscr W)\) has
unramified localization at \(w\).  Therefore its localization lies in
\[
        H^1_{\mathrm{ur}}(K_{\infty,w},\mathscr W)
        =
        H^1\left(
        \operatorname{Gal}(\bar\kappa_w/\kappa_w),
        \mathscr W^{I_w}
        \right),
\]
where \(\kappa_w\) is the residue field at \(w\) and \(I_w\) is the inertia
group. Note that
\[
        \operatorname{Gal}(\bar\kappa_w/\kappa_w)
        \simeq
        \prod_{r\neq \ell}\mathbb Z_r
\]
has no nontrivial pro-\(\ell\) quotient. Since \(\mathscr W^{I_w}\) is a
discrete \(\ell\)-primary module, it follows that
\[
        H^1\left(
        \operatorname{Gal}(\bar\kappa_w/\kappa_w),
        \mathscr W^{I_w}
        \right)=0.
\]
Consequently any class which is unramified outside \(S\) is automatically
locally trivial at the newly added primes above \(v\). This proves the result.
\end{proof}
In light of the above result, we may assume that $S$ is large enough so that $\mathscr{U}=\mathscr X\setminus S$ is affine. We shall also drop the subscript from our notation for the fine Selmer group and simply write $R(\mathscr W/K_\infty)$.

\subsection{Iwasawa modules and invariants}
\par Let \(M_1\) and $M_2$ be finitely generated and torsion
\(\Lambda\)-modules. A homomorphism of $\Lambda$-modules $M_1\rightarrow M_2$ is called a pseudo-isomorphism if both its kernel and cokernel have finitely many elements. Let $M$ be a finitely generated $\Lambda$-module. Then there is a pseudoisomorphism $M\rightarrow M'$ where
\[M':=\Lambda^r\oplus \left(\bigoplus_{i=1}^t\Lambda/(\ell^{a_i})\right)\oplus \left(\bigoplus_{j=1}^s\Lambda/(f_j(T)^{b_j})\right),
\]
where $r\geq 0$, \(a_i,b_j\geq 1\), and where each \(f_j(T)\) is a distinguished polynomial (see \cite[Ch.~13]{washington}). This means that $f_j(T)$ is a monic polynomial whose nonleading coefficients are all divisible by $\ell$. The \(\mu\)-invariant is
\[
        \mu_\Lambda(M)=\sum_i a_i.
\]The $\mu$-invariant is $0$ precisely when $t=0$, i.e., when there are no summands of the form $\Lambda/(\ell^n)$ above. 

\begin{lemma}\label{lem:lisse-nakayama}
Let \(M\) be a finitely generated \(\Lambda\)-module. Then the following are equivalent:
\begin{enumerate}
    \item $M$ is torsion as a $\Lambda$-module, and $\mu_\Lambda(M)=0$;
    \item $M/\ell M$ is finite.
\end{enumerate}
\end{lemma}

\begin{proof}
The result is an easy consequence of the above discussion. 
\end{proof}

Let
\[
        X_{\mathcal F}(\mathscr W/K_\infty)
        :=
        \operatorname{Sel}_{\mathcal F}(\mathscr W/K_\infty)^\vee
\]
be the Pontryagin dual of the Selmer group over $K_\infty$.

\begin{lemma}[Compact Nakayama lemma]
\label{lem:balister-howson-nakayama}
Let \(\mathcal R\) be a compact topological ring, and let
\(I\subseteq \mathcal R\) be a closed two-sided ideal such that
\(I^n\to 0\) as \(n\to\infty\), i.e. for every open neighbourhood \(U\) of
\(0\) in \(\mathcal R\), there exists \(n\) such that \(I^n\subseteq U\).
Let \(M\) be a compact \(\mathcal R\)-module.  If \(M/IM\) is finitely
generated as an \(\mathcal R/I\)-module, then \(M\) is finitely generated as an
\(\mathcal R\)-module.
\end{lemma}
\begin{proof}
    This result is due to Balister and Howson \cite{balisterhowson}, see also \cite[Theorem 3.6]{bandinilonghi}
\end{proof}
We endow $\Lambda$ with its profinite topology. Its maximal ideal is $\mathfrak m_\Lambda=(\ell,T)$. By a compact \(\Lambda\)-module we mean a topological \(\Lambda\)-module whose
underlying topological space is compact and Hausdorff, and for which the
\(\Lambda\)-action is continuous. Equivalently, in the profinite setting, such
a module is an inverse limit of finite discrete \(\Lambda\)-modules. 
\begin{proposition}\label{cor:lisse-mu-criterion}
Assume that $\operatorname{Sel}_{\mathcal F}(\mathscr M/K_\infty)$ is finite, then
\[
        X_{\mathcal F}(\mathscr W/K_\infty)
        :=
        \operatorname{Sel}_{\mathcal F}(\mathscr W/K_\infty)^\vee
\]
is a finitely generated torsion \(\Lambda\)-module and
\[
        \mu_\Lambda\bigl(X_{\mathcal F}(\mathscr W/K_\infty)\bigr)=0.
\]
\end{proposition}

\begin{proof}
 We set $X:=X_{\mathcal F}(\mathscr W/K_\infty)$.
Since
\(\operatorname{Sel}_{\mathcal F}(\mathscr W/K_\infty)\) is a discrete
\(\ell\)-primary group with a continuous \(\Gamma\)-action, its Pontryagin dual
\(X\) is a compact \(\Lambda\)-module. By Lemma \ref{lem:lisse-selmer-comparison}, $\operatorname{Sel}_{\mathcal F}(\mathscr W/K_\infty)[\ell]$
is finite and therefore $X/\ell X$ is finite. According to Lemma \ref{lem:lisse-nakayama} it suffices therefore to show that $X$ is finitely generated as a $\Lambda$-module. Since $X/\mathfrak m_\Lambda X$ is a quotient of \(X/\ell X\), it is finite.  Hence Lemma \ref{lem:balister-howson-nakayama}, applied with
\(I=\mathfrak m_\Lambda\), implies that \(X\) is finitely generated as a
\(\Lambda\)-module. This proves the result.
\end{proof}

\section{Greenberg Selmer groups for ordinary lisse sheaves}
\label{sec:greenberg-selmer}

\par \par The general sheaf-theoretic framework developed in this article allows one to impose a variety of local conditions and thereby construct Selmer groups beyond the standard unramified and fine Selmer groups considered above. An important example is provided by ordinary lisse sheaves, for which a suitable local filtration gives rise to Greenberg local conditions. We briefly describe this construction and illustrate it using the ordinary filtrations on symmetric powers of Tate modules of abelian varieties and their symmetric powers.
 \par Let \(K\) be a global field, let \(S\) be a finite set of places of
\(K\), and let \(\mathscr U\) denote the complement of \(S\) in the
corresponding arithmetic curve. Let $\mathscr T$
be a lisse \(\mathbb Z_\ell\)-sheaf on \(\mathscr U\), and let $T:=\mathscr T_{\bar\eta}$
be its geometric generic stalk. Thus \(T\) is a finite free \(\mathbb Z_\ell\)-module equipped with a continuous action of
\(G_S(K)\). We set
\[
        W
        :=
        T\otimes_{\mathbb Z_\ell}
        \mathbb Q_\ell/\mathbb Z_\ell.
\]

\subsection{Ordinary filtrations}

\par Let \(\Sigma_{\mathrm{ord}}\subseteq S\) be a distinguished set of
places. For every \(v\in\Sigma_{\mathrm{ord}}\), suppose there is a short exact sequence of
$G_{F_v}$-modules:
\[
        0
        \longrightarrow T_v^+
        \longrightarrow T
        \longrightarrow T_v^-
        \longrightarrow 0
\]
with \(T_v^\pm\) finite free and assume that $T_v^-$ is unramified. We refer to \(T_v^+\)
and \(T_v^-\) as the plus and minus parts of \(T\) at \(v\). We say that \(\mathscr T\) is ordinary
with respect to \(\Sigma_{\mathrm{ord}}\) if it is ordinary at every
\(v\in\Sigma_{\mathrm{ord}}\). Tensoring with
\(\mathbb Q_\ell/\mathbb Z_\ell\) gives an exact sequence
\[
        0
        \longrightarrow W_v^+
        \longrightarrow W
        \longrightarrow W_v^-
        \longrightarrow 0,
\]
where
\[
        W_v^\pm
        :=
        T_v^\pm\otimes_{\mathbb Z_\ell}
        \mathbb Q_\ell/\mathbb Z_\ell.
\]

\subsection{Greenberg local conditions}

\par Let \(L/K\) be a finite extension contained in \(K_S\), and let
\(w\) be a place of \(L\) lying above \(v\). If
\(v\in\Sigma_{\mathrm{ord}}\), we regard the ordinary filtration at
\(v\) as a filtration of \(D_w\)-modules by restriction.

\begin{definition}
\label{def:greenberg-local-condition}
For \(w\mid v\), with \(v\in\Sigma_{\mathrm{ord}}\), the Greenberg local
condition is
\[
        H^1_{\mathrm{Gr}}(L_w,W)
        :=
        \ker\left\{
        H^1(L_w,W)
        \longrightarrow
        H^1(I_w,W_v^-)
        \right\}.
\]
Equivalently, it is the inverse image of the unramified subgroup
\[
        H^1_{\mathrm{ur}}(L_w,W_v^-)
        :=
        \ker\left\{
        H^1(L_w,W_v^-)
        \longrightarrow
        H^1(I_w,W_v^-)
        \right\}
\]
under the natural map
\[
        H^1(L_w,W)
        \longrightarrow
        H^1(L_w,W_v^-).
\]
\end{definition}

\par At a place \(w\) not lying above
\(\Sigma_{\mathrm{ord}}\), we use the unramified local condition
\[
        H^1_{\mathrm{ur}}(L_w,W)
        :=
        \ker\left\{
        H^1(L_w,W)
        \longrightarrow
        H^1(I_w,W)
        \right\}.
\]
More generally, the unramified condition at such places may be replaced
by any compatible collection of local conditions.

\begin{definition}
\label{def:greenberg-selmer}
The Greenberg Selmer group of \(W\) over \(L\) is
\[
\begin{split}
        \operatorname{Sel}_{\mathrm{Gr}}(W/L)
        :=
        \ker\Bigg\{
        H^1(G_S(L),W)
        &\longrightarrow
        \prod_{\substack{w\in S_L\\
                         w\mid\Sigma_{\mathrm{ord}}}}
        \frac{H^1(L_w,W)}
             {H^1_{\mathrm{Gr}}(L_w,W)}
\\
        &\qquad\times
        \prod_{\substack{w\in S_L\\
                         w\nmid\Sigma_{\mathrm{ord}}}}
        \frac{H^1(L_w,W)}
             {H^1_{\mathrm{ur}}(L_w,W)}
        \Bigg\}.
\end{split}
\]
\end{definition}

\par If \(K_\infty/K\) is an infinite Galois extension contained in
\(K_S\), we define
\[
        \operatorname{Sel}_{\mathrm{Gr}}(W/K_\infty)
        :=
        \varinjlim_{L\subseteq K_\infty}
        \operatorname{Sel}_{\mathrm{Gr}}(W/L),
\]
where \(L\) ranges over the finite extensions of \(K\) contained in
\(K_\infty\). Its Pontryagin dual is denoted by
\[
        X_{\mathrm{Gr}}(T/K_\infty)
        :=
        \operatorname{Sel}_{\mathrm{Gr}}(W/K_\infty)^\vee.
\]

\subsection{Ordinary abelian varieties}

\par Let \(A/K\) be an abelian variety of dimension \(g\), and suppose
that \(A\) has good ordinary reduction at every place
\(v\in\Sigma_{\mathrm{ord}}\) lying above \(\ell\). The
connected--\'etale sequence of its \(\ell\)-divisible group induces a
\(D_v\)-stable exact sequence
\[
        0
        \longrightarrow T_{\ell,v}^+(A)
        \longrightarrow T_\ell(A)
        \longrightarrow T_{\ell,v}^-(A)
        \longrightarrow 0.
\]
Here
\[
        \operatorname{rank}_{\mathbb Z_\ell}
        T_{\ell,v}^+(A)
        =
        \operatorname{rank}_{\mathbb Z_\ell}
        T_{\ell,v}^-(A)
        =g,
\]
and the action of \(I_v\) on \(T_{\ell,v}^-(A)\) is trivial. Thus
\(T_\ell(A)\) is ordinary in the above sense.

\par Writing
\[
        A[\ell^\infty]
        =
        T_\ell(A)\otimes_{\mathbb Z_\ell}
        \mathbb Q_\ell/\mathbb Z_\ell,
\]
the ordinary quotient is
\[
        A[\ell^\infty]_v^-
        :=
        T_{\ell,v}^-(A)
        \otimes_{\mathbb Z_\ell}
        \mathbb Q_\ell/\mathbb Z_\ell.
\]
The Greenberg local condition at \(w\mid v\) is therefore
\[
        H^1_{\mathrm{Gr}}
        \bigl(L_w,A[\ell^\infty]\bigr)
        =
        \ker\left\{
        H^1\bigl(L_w,A[\ell^\infty]\bigr)
        \longrightarrow
        H^1\bigl(I_w,A[\ell^\infty]_v^-\bigr)
        \right\}.
\]

\subsection{Symmetric powers}

\par We now describe the induced ordinary filtration on symmetric
powers. Let \(m\geq 1\) and \(r\in\mathbb Z\), and put
\[
        T_{m,r}
        :=
        \operatorname{Sym}^m T_\ell(A)(r).
\]
The quotient map
\[
        T_\ell(A)
        \longrightarrow
        T_{\ell,v}^-(A)
\]
induces a surjection
\[
        \operatorname{Sym}^m T_\ell(A)(r)
        \longrightarrow
        \operatorname{Sym}^m T_{\ell,v}^-(A)(r).
\]
We define
\[
        T_{m,r,v}^-
        :=
        \operatorname{Sym}^m T_{\ell,v}^-(A)(r)
\]
and
\[
        T_{m,r,v}^+
        :=
        \ker\left\{
        \operatorname{Sym}^m T_\ell(A)(r)
        \longrightarrow
        \operatorname{Sym}^m T_{\ell,v}^-(A)(r)
        \right\}.
\]
This gives an exact sequence
\[
        0
        \longrightarrow T_{m,r,v}^+
        \longrightarrow T_{m,r}
        \longrightarrow T_{m,r,v}^-
        \longrightarrow 0.
\]
Consequently, every symmetric power \(T_{m,r}\) is ordinary with
respect to the induced filtration.

\par Set
\[
        W_{m,r}
        :=
        T_{m,r}\otimes_{\mathbb Z_\ell}
        \mathbb Q_\ell/\mathbb Z_\ell
\]
and
\[
        W_{m,r,v}^-
        :=
        T_{m,r,v}^-\otimes_{\mathbb Z_\ell}
        \mathbb Q_\ell/\mathbb Z_\ell.
\]
The Greenberg local condition associated with the symmetric power is
\[
        H^1_{\mathrm{Gr}}(L_w,W_{m,r})
        :=
        \ker\left\{
        H^1(L_w,W_{m,r})
        \longrightarrow
        H^1(I_w,W_{m,r,v}^-)
        \right\}.
\]
The corresponding global Selmer group is denoted by $\operatorname{Sel}_{\mathrm{Gr}}
        \bigl(W_{m,r}/L\bigr)$.
        \subsection{Adjoint Selmer groups}

\par Let \(\mathscr T\) be a rank-two lisse
\(\mathbb Z_\ell\)-sheaf. The adjoint sheaf is
\[
        \operatorname{Ad}(\mathscr T)
        :=
        \operatorname{End}_{\mathbb Z_\ell}(\mathscr T),
\]
with Galois action given by conjugation. Its trace-zero subrepresentation
is denoted by
\[
        \operatorname{Ad}^0(\mathscr T)
        :=
        \ker\left(
        \operatorname{tr}:
        \operatorname{End}_{\mathbb Z_\ell}(\mathscr T)
        \longrightarrow
        \mathbb Z_\ell
        \right).
\]
If \(\ell\neq 2\), then
\[
        \operatorname{Ad}^0(\mathscr T)
        \simeq
        \operatorname{Sym}^2\mathscr T
        \otimes_{\mathbb Z_\ell}
        \det(\mathscr T)^{-1}.
\]
Thus the Greenberg Selmer group associated with the trace-zero adjoint
sheaf is a special case of the symmetric-power Selmer groups considered
above. In particular, for an elliptic curve \(E\),
\[
        \operatorname{Ad}^0(T_\ell(E))
        \simeq
        \operatorname{Sym}^2T_\ell(E)(-1).
\]

\section{Vanishing of the $\mu$-invariant}

\par We first record a simple cohomological lemma.

\begin{lemma}\label{lem:lisse-no-pro-l}
Let \(N\) be a finite \(\ell\)-primary discrete \(G_{k_\infty}\)-module.  Then
\[
        H^a(G_{k_\infty},N)=0
        \qquad
        \text{for all }a>0.
\]
\end{lemma}

\begin{proof}
Since \(k_\infty/k\) is the constant \(\mathbb Z_\ell\)-extension, we have
\[
        G_{k_\infty}
        =
        \operatorname{Gal}(\bar k/k_\infty)
        \simeq
        \prod_{r\neq \ell}\mathbb Z_r.
\]
Thus \(G_{k_\infty}\) is a pro-prime-to-\(\ell\) group.  In particular, every
finite quotient \(Q\) of \(G_{k_\infty}\) has order prime to \(\ell\).
\par Since \(N\) is finite and discrete, continuous cochains of
\(G_{k_\infty}\) with values in \(N\) factor through finite quotients of
\(G_{k_\infty}\). Hence
\[
        H^a(G_{k_\infty},N)
        =
        \varinjlim_{U}
        H^a(G_{k_\infty}/U,N),
\]
where \(U\) ranges over sufficiently small open normal subgroups of
\(G_{k_\infty}\) acting trivially on \(N\).

\par Let $Q=G_{k_\infty}/U$ be one such finite quotient.  Then \(|Q|\) is prime to \(\ell\).  By the
restriction--corestriction formula, if \(H\subseteq Q\) is a subgroup, then
\[
        \operatorname{Cor}_{H}^{Q}\circ
        \operatorname{Res}_{H}^{Q}
        =
        [Q:H]
\]
on \(H^a(Q,N)\); see Serre \cite[p.~12]{Serre}.  Taking \(H=\{1\}\), and using
that
\[
        H^a(\{1\},N)=0
        \qquad
        \text{for }a>0,
\]
we find that multiplication by \(|Q|\) annihilates \(H^a(Q,N)\).

On the other hand, \(N\) is \(\ell\)-primary and \(|Q|\) is prime to \(\ell\).
Therefore multiplication by \(|Q|\) is an automorphism of \(N\), and hence of
\(H^a(Q,N)\). Hence
\[
        H^a(Q,N)=0
        \qquad
        \text{for all }a>0.
\]
Passing to the direct limit over finite quotients \(Q\), we get
\[
        H^a(G_{k_\infty},N)=0
        \qquad
        \text{for all }a>0.
\]
\end{proof}
\begin{theorem}
\label{thm:stacks-torsion-curves}
Let \(k_0\) be an algebraically closed field, and let \(Y\) be a separated
scheme of finite type over \(k_0\) with
\[
        \dim Y\leq 1.
\]
Let \(\mathscr F\) be a torsion abelian sheaf on \(Y_{\mathrm{et}}\).  Then the
following assertions hold:
\begin{enumerate}
    \item One has
    \[
            H^i_{\mathrm{et}}(Y,\mathscr F)=0
            \qquad
            \text{for all } i>2.
    \]

    \item If \(Y\) is affine, then
    \[
            H^i_{\mathrm{et}}(Y,\mathscr F)=0
            \qquad
            \text{for all } i>1.
    \]

    \item If \(p=\operatorname{char}(k_0)>0\) and \(\mathscr F\) is
    \(p\)-primary torsion, then
    \[
            H^i_{\mathrm{et}}(Y,\mathscr F)=0
            \qquad
            \text{for all } i>1.
    \]

    \item If \(\mathscr F\) is constructible and torsion prime to
    \(\operatorname{char}(k_0)\), then $H^i_{\mathrm{et}}(Y,\mathscr F)$ is finite for every \(i\geq 0\).

    \item If \(Y\) is proper over \(k_0\) and \(\mathscr F\) is constructible,
    then $H^i_{\mathrm{et}}(Y,\mathscr F)$ is finite for every \(i\geq 0\).
\end{enumerate}
\end{theorem}
\begin{proof}
    The result follows from \cite[Theorem 59.83.10]{Stacks}.
\end{proof}

\begin{proposition}\label{prop:lisse-residual-finite}
Let \(\mathscr M\) be a constructible \(\mathbb F_\ell\)-sheaf on
\(\mathscr U\). Set $\mathscr U_\infty=\mathscr U\times_k k_\infty$ and $\mathscr U_{\bar k}=\mathscr U\times_k \bar k$. Then, for every \(i\geq 0\), there is a canonical isomorphism
\[
        H^i_{\mathrm{et}}(\mathscr U_\infty,\mathscr M)
        \simeq
        H^i_{\mathrm{et}}(\mathscr U_{\bar k},\mathscr M)^{G_{k_\infty}}.
\]
In particular, $H^i_{\mathrm{et}}(\mathscr U_\infty,\mathscr M)$ is finite for every \(i\geq 0\).
\end{proposition}

\begin{proof}

The morphism
\[
        \mathscr U_{\bar k}\longrightarrow \mathscr U_\infty
\]
is the Galois covering obtained by extending constants from \(k_\infty\) to
\(\bar k\), with Galois group \(G_{k_\infty}\). The Hochschild--Serre spectral
sequence gives
\[
E_2^{a,b}
=
H^a\left(
G_{k_\infty},
H^b_{\mathrm{et}}(\mathscr U_{\bar k},\mathscr M)
\right)
\Longrightarrow
H^{a+b}_{\mathrm{et}}(\mathscr U_\infty,\mathscr M).
\]Since
\(\mathscr U_{\bar k}\) is a curve over the algebraically closed field
\(\bar k\), and since \(\mathscr M\) is constructible with torsion
coefficients prime to \(p\), Theorem \ref{thm:stacks-torsion-curves} part (4) implies that $H^b_{\mathrm{et}}(\mathscr U_{\bar k},\mathscr M)$
is a finite-dimensional \(\mathbb F_\ell\)-vector space for every \(b\). By
Lemma \ref{lem:lisse-no-pro-l}, we have
\[
        H^a\left(
        G_{k_\infty},
        H^b_{\mathrm{et}}(\mathscr U_{\bar k},\mathscr M)
        \right)
        =
        0
        \qquad
        \text{for all }a>0.
\]
Therefore the Hochschild--Serre spectral sequence has no nonzero terms away
from the column \(a=0\).  It degenerates at \(E_2\), and gives canonical
isomorphisms
\[
        H^i_{\mathrm{et}}(\mathscr U_\infty,\mathscr M)
        \simeq
        H^0\left(
        G_{k_\infty},
        H^i_{\mathrm{et}}(\mathscr U_{\bar k},\mathscr M)
        \right).
\]
By part (4) of Theorem \ref{thm:stacks-torsion-curves}, $H^i_{\mathrm{et}}(\mathscr U_{\bar k},\mathscr M)^{G_{k_\infty}}$ is finite and the result follows.
\end{proof}

\begin{theorem}\label{thm:lisse-selmer-mu-zero}
Let \(K\) be a global function field of characteristic \(p>0\), with field of
constants \(k=\mathbb F_q\), and let \(\ell\neq p\) be a prime number.  Let
\(\mathscr U\subseteq\mathscr X\) be a nonempty open affine subscheme, and let
\(\mathscr T\) be a lisse \(\mathbb Z_\ell\)-sheaf on \(\mathscr U\) for which \(T\) is free over \(\mathbb Z_\ell\). Put
\[
        \mathscr W=\mathscr T\otimes_{\mathbb Z_\ell}\mathbb Q_\ell/\mathbb Z_\ell,
        \qquad
        \mathscr M=\mathscr T/\ell\mathscr T.
\]
Let \(\mathcal F\) be a Selmer structure on \(\mathscr W\). Then
\[
        X_{\mathcal F}(\mathscr W/K_\infty)
        :=
        \operatorname{Sel}_{\mathcal F}(\mathscr W/K_\infty)^\vee
\]
is a finitely generated and torsion \(\Lambda\)-module and
\[
        \mu_\Lambda\bigl(X_{\mathcal F}(\mathscr W/K_\infty)\bigr)=0.
\]
\end{theorem}

\begin{proof}
We define a Selmer structure on $\mathscr{W}$ so that Assumption \ref{assumption:cartesian-local-conditions} is satisfied. By Proposition \ref{prop:lisse-residual-finite}, the group $H^1_{\mathrm{et}}(\mathscr U_\infty,\mathscr M)$
is finite. Since $\operatorname{Sel}_{\mathcal F}(\mathscr M/K_\infty)$ is contained in $H^1_{\mathrm{et}}(\mathscr U_\infty,\mathscr M)$, it follows that $\operatorname{Sel}_{\mathcal F}(\mathscr M/K_\infty)$ is finite as well. The result follows from Proposition \ref{cor:lisse-mu-criterion}.
\end{proof}

We end this section with proof of an analogue of the weak Leopoldt conjecture. Let \(K\) be a global function field of characteristic \(p>0\), with field of
constants \(k=\mathbb F_q\), and let \(\ell\neq p\).  Let
\(\mathscr X/k\) be the smooth, projective, geometrically integral curve with
function field \(K\).  Let $\mathscr U\subseteq \mathscr X$ be a nonempty open and affine subscheme, and write \(S=\mathscr X\setminus\mathscr U\).
Thus \(S\) is a finite set of closed points of \(\mathscr X\).

Let \(\mathscr T\) be a lisse \(\mathbb Z_\ell\)-sheaf on \(\mathscr U\) for which \(T\) is finitely generated and free over \(\mathbb Z_\ell\). Recall that 
\[
        \mathscr W
        =
        \mathscr T\otimes_{\mathbb Z_\ell}\mathbb Q_\ell/\mathbb Z_\ell.
\]
\begin{definition}\label{def:weak-leopoldt-lisse}
We say
that weak Leopoldt holds for the pair \((\mathscr T,K_\infty)\) if
\[
        H^2\bigl(K_S/K_\infty,\mathscr W\bigr)=0.
\]
\end{definition}

\begin{theorem}\label{thm:weak-leopoldt-constant}
The weak Leopoldt conjecture in the sense of the above definition holds.
\end{theorem}

\begin{proof}
Since continuous Galois cohomology with discrete coefficients commutes with
filtered direct limits, it is enough to prove that
\[
        H^2\bigl(G_S(K_\infty),\mathscr T_n\bigr)=0
\]
for every \(n\). The cohomology group $H^2\bigl(G_S(K_\infty),\mathscr T_n\bigr)$ identifies with $H^2_{\mathrm{et}}(\mathscr U_\infty,\mathscr T_n)$. The result then follows from Proposition \ref{prop:lisse-residual-finite} and part (2) of Theorem \ref{thm:stacks-torsion-curves}.
\end{proof}

\section{Deformation rings and unobstructedness}

\par Let \(K_\infty/K\) be a \(\mathbb Z_\ell\)-extension of \(K\). Given an algebraic extension $F/K$ contained in $K_S$, we set $G_S(F):=\operatorname{Gal}(K_S/F)$. Let \(\mathbb F\) be a finite field of characteristic \(\ell\), and let $\mathcal O=W(\mathbb F)$
be the ring of Witt vectors of \(\mathbb F\). In other words, \(\mathcal O\) is the ring
of integers of the unramified extension of \(\mathbb Q_\ell\) with residue
field \(\mathbb F\). We write \(\varpi\) for a uniformizer of \(\mathcal O\).

\par Let $\bar r:G_{K,S}\longrightarrow \operatorname{GL}_d(\mathbb F)$ be a continuous Galois representation. We set $\bar r_\infty:=\bar r|_{G_S(K_\infty)}$ and $\bar r_n:=\bar r|_{G_S(K_n)}$.
Let $V_{\bar r}:=\mathbb F^d$ be the underlying \(\mathbb F\)-vector space of \(\bar r\). The adjoint
representation attached to \(\bar r\) is defined as follows:
\[
        \operatorname{Ad}\bar r
        :=
        \operatorname{End}_{\mathbb F}(V_{\bar r}),
\]
with \(G_S(K)\)-action given by conjugation 
\[g\cdot X
        :=
        \bar r(g)X\bar r(g)^{-1}\] for $g\in G_{K,S}$ and $X\in\operatorname{End}_{\mathbb F}(V_{\bar r})$. If one works with deformations of fixed determinant, then
\(\operatorname{Ad}\bar r\) is replaced throughout by the trace-zero adjoint
representation
\[
        \operatorname{Ad}^0\bar r
        :=
        \{X\in \operatorname{End}_{\mathbb F}(V_{\bar r})\mid \operatorname{tr}(X)=0\}.
\]
We have a decomposition of \(G_{K,S}\)-modules
\[
        \operatorname{Ad}\bar r
        \simeq
        \operatorname{Ad}^0\bar r\oplus \mathbb F\cdot \operatorname{Id}.
\]
Note that the Galois action on \(\mathbb F\cdot \operatorname{Id}\) is the trivial one.

\par Let \(\mathcal C_{\mathcal O}\) denote the category of Artinian local
\(\mathcal O\)-algebras $(R, \mathfrak{m}_R)$ equipped with the unique isomorphism of $\cO$-algebras $R/\mathfrak m_R\simeq \mathbb F$.
Morphisms in \(\mathcal C_{\mathcal O}\) are local \(\mathcal O\)-algebra
homomorphisms $\varphi: (R_1, \mathfrak{m}_{R_1})\rightarrow(R_2, \mathfrak{m}_{R_2}) $ such that the reduced map 
\[\bar{\varphi}: R_1/\mathfrak{m}_{R_1}\rightarrow R_2/\mathfrak{m}_{R_2}\]
induces the identity on \(\mathbb F\).  We also let
\(\operatorname{CNL}_{\mathcal O}\) denote the category of complete noetherian
local \(\mathcal O\)-algebras with residue field \(\mathbb F\). An object of
\(\operatorname{CNL}_{\mathcal O}\) may be regarded as a projective limit of
objects of \(\mathcal C_{\mathcal O}\).
\begin{definition}For \(R\in\mathcal C_{\mathcal O}\) and $\ast\in \Z_{\geq 1}\cup \{\infty\}$ a lift of \(\bar r_\ast\) to \(R\)
is a continuous representation
\[
        r_R:G_S(K_\ast)\longrightarrow \operatorname{GL}_d(R)
\]
such that the reduction of \(r_R\) modulo \(\mathfrak m_R\) is equal to
\(\bar r_\ast\). Two lifts \(r_R\) and \(r_R'\) are said to be strictly
equivalent if there exists a matrix $A\in 1+M_d(\mathfrak m_R)$ such that $r_R'(g)=A r_R(g) A^{-1}$ for all $g\in G_S(K_\ast)$. 
The deformation functor
\[
        \operatorname{Def}_{\bar r_\ast}:\mathcal C_{\mathcal O}
        \longrightarrow \operatorname{Sets}
\]
is defined by sending \(R\) to the set of strict equivalence classes of lifts
of \(\bar r_\ast\) to \(R\).
\end{definition}

\par We shall consider the framed deformation functor
\[
        \operatorname{Def}_{\bar r_\ast}^{\square}:
        \mathcal C_{\mathcal O}\longrightarrow \operatorname{Sets}
\]
sends \(R\) to the set of all lifts
\[
        r_R:G_S(K_\ast)\longrightarrow \operatorname{GL}_d(R)
\]
of \(\bar r_\ast\), without quotienting by strict equivalence.

\par In the deformation theory of Galois representations, the study of infinitesimal lifts plays a major role. They are controlled by the cohomology of the adjoint
representation. This is seen naturally when we consider deformations of $\bar{r}$ to the dual numbers $\mathbb F[\varepsilon]:=\mathbb F[t]/(t^2)$. Here, $\F[\epsilon]$ is an $\cO$-algebra and the structure morphism factors as follows
\[\cO\twoheadrightarrow \F\hookrightarrow \F[\epsilon],\] where the first map is the reduction map and the second is the natural inclusion. A lift of \(\bar r_\ast\) to \(\mathbb F[\varepsilon]\) is a
representation of the form
\[
        r_c(g)
        =
        \bigl(1+\varepsilon c(g)\bigr)\bar r_\ast(g),
\]
where $c:G_S(K_\ast)\longrightarrow \operatorname{Ad}\bar r$
is a continuous map. The condition that \(r_c\) be a homomorphism is
equivalent to the cocycle condition
\[c(gh)=c(g)+g\cdot c(h)\] for all $g,h\in G_S(K_\ast)$,
where \(g\cdot c(h)\) denotes the conjugation action through \(\bar r_\infty\).
Thus lifts to dual numbers are parametrized by continuous \(1\)-cocycles $Z^1\bigl(G_S(K_\ast),\operatorname{Ad}\bar r\bigr)$.

\par Strict equivalence corresponds to quotienting by coboundaries. Indeed,
conjugation by a matrix of the form $1+\varepsilon X$ with $X\in \operatorname{Ad}\bar r$ changes \(c\) to
\[
        c'(g)=c(g)+X-g\cdot X.
\]
Thus \(c\) and \(c'\) define strictly equivalent deformations if and only if
they differ by a \(1\)-coboundary.  Therefore the tangent space of the unframed
deformation functor is canonically identified with
\[
        \operatorname{Def}_{\bar r_\ast}(\mathbb F[\varepsilon])
        \simeq
        H^1\bigl(G_S(K_\ast),\operatorname{Ad}\bar r\bigr).
\]
For the framed deformation functor, no quotient by strict equivalence is taken,
and hence
\[
        \operatorname{Def}^{\square}_{\bar r_\ast}(\mathbb F[\varepsilon])
        \simeq
        Z^1\bigl(G_S(K_\ast),\operatorname{Ad}\bar r\bigr).
\]
If the determinant is fixed, then the same discussion gives
\[
        \operatorname{Def}_{\bar r_\ast}^{\det=\psi}(\mathbb F[\varepsilon])
        \simeq
        H^1\bigl(G_S(K_\ast),\operatorname{Ad}^0\bar r\bigr),
\]
where \(\psi\) is the prescribed determinant lifting
\(\det(\bar r_\ast)\).
\par We recall the standard representability statement for deformation functors
at finite level.  For background, see Mazur \cite{Mazurdeforming} and Kisin
\cite[Proposition 1.2.1]{Kisin}.  Let \(n\geq 1\), and set
\[
        \bar r_n:=\bar r|_{G_S(K_n)}.
\]
Since \(K_n/K\) is a finite extension and \(S\) is finite, the profinite group
\(G_S(K_n)\) satisfies Mazur's finiteness condition \(\Phi_\ell\).  Therefore
the standard deformation-theoretic representability results apply to
\(\bar r_n\).

\par The framed deformation functor
\(\operatorname{Def}_{\bar r_n}^{\square}\) is pro-representable.  Thus there
exists a complete noetherian local \(\mathcal O\)-algebra
\[
        R_{\bar r_n}^{\square}\in \operatorname{CNL}_{\mathcal O}
\]
and a universal framed lift
\[
        r_{n,\mathrm{univ}}^{\square}:
        G_S(K_n)\longrightarrow
        \operatorname{GL}_d(R_{\bar r_n}^{\square})
\]
with the following universal property: for every
\(R\in\mathcal C_{\mathcal O}\), composition with
\(r_{n,\mathrm{univ}}^{\square}\) induces a functorial bijection
\[
        \operatorname{Hom}_{\operatorname{CNL}_{\mathcal O}}
        (R_{\bar r_n}^{\square},R)
        \simeq
        \operatorname{Def}_{\bar r_n}^{\square}(R).
\]
Here the word framed means that one remembers a choice of basis lifting the
fixed residual basis of \(V_{\bar r}\).  Equivalently,
\(\operatorname{Def}_{\bar r_n}^{\square}(R)\) consists of lifts of
\(\bar r_n\) to \(R\), without quotienting by conjugation.

\par The unframed deformation functor is more subtle because one quotients
lifts by strict equivalence, namely by conjugation by matrices congruent to the
identity modulo the maximal ideal of \(R\).  A standard sufficient hypothesis
for pro-representability is the scalar-endomorphism condition
\[
        \operatorname{End}_{\mathbb F[G_S(K_n)]}(V_{\bar r})=\mathbb F.
\]
Assume this condition.  Then the functor
\(\operatorname{Def}_{\bar r_n}\) is pro-representable by a complete noetherian
local \(\mathcal O\)-algebra
\[
        R_{\bar r_n}\in \operatorname{CNL}_{\mathcal O}.
\]
Equivalently, there is a universal lift
\[
        r_{n,\mathrm{univ}}:
        G_S(K_n)\longrightarrow \operatorname{GL}_d(R_{\bar r_n}),
\]
well-defined up to strict equivalence, such that for every
\(R\in\mathcal C_{\mathcal O}\) one has a functorial bijection
\[
        \operatorname{Hom}_{\operatorname{CNL}_{\mathcal O}}
        (R_{\bar r_n},R)
        \simeq
        \operatorname{Def}_{\bar r_n}(R).
\]
The scalar-endomorphism hypothesis ensures that the residual representation
has no non-scalar automorphisms.  This is precisely the condition needed for
the functor of strict equivalence classes to satisfy Schlessinger's
representability criterion.

\par The same discussion applies to fixed-determinant deformations.  Fix a
continuous character
\[
        \psi:G_S(K_n)\longrightarrow \mathcal O^\times
\]
lifting \(\det(\bar r_n)\).  For \(R\in\mathcal C_{\mathcal O}\), let
\(\psi_R\) denote the character obtained from \(\psi\) by composing
\(\mathcal O^\times\to R^\times\).  The fixed-determinant deformation functor
\(\operatorname{Def}_{\bar r_n}^{\det=\psi}\) assigns to \(R\) the set of
strict equivalence classes of lifts
\[
        r_R:G_S(K_n)\longrightarrow \operatorname{GL}_d(R)
\]
whose determinant is \(\psi_R\).  Under the same scalar-endomorphism hypothesis
\[
        \operatorname{End}_{\mathbb F[G_S(K_n)]}(V_{\bar r})=\mathbb F,
\]
this functor is pro-representable by a complete noetherian local
\(\mathcal O\)-algebra $R_{\bar r_n}^{\det=\psi}$. It may be described as the quotient of \(R_{\bar r_n}\) obtained by imposing
the equations
\[
        \det(r_{n,\mathrm{univ}})=\psi.
\]
For the framed fixed-determinant functor, no scalar-endomorphism hypothesis is
needed.  It is represented by the corresponding determinant quotient of the
framed deformation ring \(R_{\bar r_n}^{\square}\). For ease of notation, we set $G_n:=G_S(K_n)$ and $G_\infty:=G_S(K_\infty)$. Let \(\widehat{\mathcal C}_{\mathcal O}\) be the category of complete
local \(\mathcal O\)-algebras with residue field \(\mathbb F\) (not necessarily
noetherian). Assume that, for every \(n\geq 0\), the deformation functor $\operatorname{Def}_{\bar r_n}$ is pro-representable by a complete noetherian local \(\mathcal O\)-algebra
\(R_{\bar r_n}\). The restriction maps
\[
        G_{n+1}\subseteq G_n
\]
induce natural transformations
\[
        \operatorname{Def}_{\bar r_n}
        \longrightarrow
        \operatorname{Def}_{\bar r_{n+1}},
\]
and hence transition maps
\[
        R_{\bar r_{n+1}}\longrightarrow R_{\bar r_n}.
\]
Set
\[
        R_{\bar r_\infty}
        :=
        \varprojlim_n R_{\bar r_n}.
\]

\begin{proposition}\label{prop:deformation-ring-infinite-layer}
With respect to notation above, the deformation functor of
\(\bar r_\infty\) is prorepresentable by \(R_{\bar r_\infty}\). More precisely,
for every \(A\in\mathcal C_{\mathcal O}\), there is a natural bijection
\[
        \operatorname{Hom}_{\widehat{\mathcal C}_{\mathcal O}}
        \left(R_{\bar r_\infty},A\right)
        \simeq
        \operatorname{Def}_{\bar r_\infty}(A).
\]
\end{proposition}

\begin{proof}
The proof follows from a standard argument \cite[Lemma 1.3]{BurungaleClozel}, also see \cite[Section 3.3]{ClozelFormes}.
\end{proof}

\par The obstruction theory is described using small extensions.  A small
extension in \(\mathcal C_{\mathcal O}\) is a surjection $R'\twoheadrightarrow R$
whose kernel \(I\) is principal and satisfies
\[
        I\mathfrak m_{R'}=0.
\]
In particular, \(I^2=0\), and \(I\) is naturally an \(\mathbb F\)-vector space. Let $t$ be a generator of $I$. Suppose that
\[
        r_R:G_S(K_\infty)\longrightarrow \operatorname{GL}_d(R)
\]
is a lift of \(\bar r_\infty\).  Choose for each $g\in G_S(K_\infty)$, a lift
\[
        \widetilde r(g)\in \operatorname{GL}_d(R')
\]
of \(r_R(g)\). These choices need not define a homomorphism. The defect is
measured by
\[
        \widetilde r(g)\widetilde r(h)\widetilde r(gh)^{-1}
        \in
        1+M_d(I).
\]
Since \(I^2=0\), we may write this element uniquely in the form
\[
        1+\alpha(g,h),
        \qquad
        \alpha(g,h)\in \operatorname{Ad}\bar r\otimes_{\mathbb F} I.
\]
The function
\[
        \alpha:G_S(K_\infty)\times G_S(K_\infty)
        \longrightarrow
        \operatorname{Ad}\bar r\otimes_{\mathbb F} I
\]
is a continuous \(2\)-cocycle.  Its cohomology class
\[
        [\alpha]\in
        H^2\bigl(G_S(K_\infty),\operatorname{Ad}\bar r\bigr)
        \otimes_{\mathbb F} I
\]
is independent of the choices of the set-theoretic lifts
\(\widetilde r(g)\).  This class is the cohomological obstruction to lifting \(r_R\) to
\(R'\).  More precisely, \(r_R\) admits a lift to \(R'\) if and only if
\[
        [\alpha]=0.
\]
If the determinant is fixed, the obstruction lies in
\[
        H^2\bigl(G_S(K_\infty),\operatorname{Ad}^0\bar r\bigr)
        \otimes_{\mathbb F} I.
\]
\par The preceding discussion gives the unobstructedness criterion for infinitesimal liftings, which is to say that if
\[
        H^2\bigl(G_S(K_\infty),\operatorname{Ad}\bar r\bigr)=0,
\]
then every lift over an object \(R\in\mathcal C_{\mathcal O}\) lifts across
every small extension
\[
        R'\twoheadrightarrow R.
\]
Equivalently, the framed deformation functor is formally smooth over
\(\mathcal O\).  In this case the universal framed deformation ring is a formal
power series ring over \(\mathcal O\),
\[
        R_{\bar r_\infty}^{\square}
        \simeq
        \mathcal O[[x_1,\dots,x_m]],
\]
where
\[
        m:=\dim_{\mathbb F}
        Z^1\bigl(G_S(K_\infty),\operatorname{Ad}\bar r\bigr)
\]
provided this dimension is finite.  If, in addition,
\[
        \operatorname{End}_{\mathbb F[G_S(K_\infty)]}(V_{\bar r})=\mathbb F,
\]
then the unframed deformation functor is pro-representable, and
\[
        H^2\bigl(G_S(K_\infty),\operatorname{Ad}\bar r\bigr)=0
\]
implies that $R_{\bar r_\infty}$ is formally smooth over
\(\mathcal O\). Further, if 
\[
        m':=\dim_{\mathbb F}
        H^1\bigl(G_S(K_\infty),\operatorname{Ad}\bar r\bigr)
\]
is finite, then it is a formal power series ring over $\cO$ in $m'$ variables.

\begin{definition}
We say that $\bar r_\infty$ is \emph{unobstructed} if
\[
        H^2\bigl(G_S(K_\infty),\operatorname{Ad}\bar r\bigr)=0.
\]
\end{definition}

\begin{theorem}\label{thm:constant-tower-deformation-unobstructed}
Let $K=k(\mathscr{X})$ be a global function field of characteristic $p>0$ and $\ell\neq p$ be a prime number. Let $K_\infty/K$ be the $\Z_\ell$-extension of $K$ and $S$ be a finite set of closed points on $\mathscr{X}$ such that $\mathscr{U}:=\mathscr{X}\setminus S$ is affine. Let $\F$ be a finite field of characteristic $\ell$ and $\cO$ be its ring of Witt vectors. Given a Galois representation $\bar{r}:\op{G}_S(K)\rightarrow \op{GL}_d(\F)$, the associated framed deformation ring is a formal power series ring in finitely many variables:
\[R_{\bar{r}_\infty}^{\square}\simeq \cO[[x_1, \dots, x_m]].\] Moreover, if 
\[\operatorname{End}_{\mathbb F[G_S(K_\infty)]}(V_{\bar r})=\mathbb F\] then the (unframed) deformation ring is a formal power series in finitely many variables
\[R_{\bar{r}_\infty}\simeq \cO[[y_1, \dots, y_{m'}]].\] 
\end{theorem}

\begin{proof}
Let $\operatorname{Ad}\bar r
        =
        \operatorname{End}_{\mathbb F}(V_{\bar r})$ with \(G_S(K)\)-action by conjugation.  Since \(\bar r\) is unramified outside
\(S\), the representation \(\operatorname{Ad}\bar r\) defines a constructible
\(\mathbb F\)-sheaf on \(\mathscr U=\mathscr X\setminus S\). By Proposition \ref{prop:lisse-residual-finite}, applied to the constructible
\(\mathbb F\)-sheaf \(\operatorname{Ad}\bar r\), we have canonical
isomorphisms
\[
        H^i(G_\infty,\operatorname{Ad}\bar r)
        =
        H^i_{\mathrm{\acute et}}
        (\mathscr U_\infty,\operatorname{Ad}\bar r)
        \simeq
        H^i_{\mathrm{\acute et}}
        (\mathscr U_{\bar k},\operatorname{Ad}\bar r)^{G_{k_\infty}}
\]
for all \(i\geq 0\). Since \(\mathscr U\) is
affine, the curve \(\mathscr U_{\bar k}\) is affine over the algebraically
closed field \(\bar k\).  Hence, by Theorem
\ref{thm:stacks-torsion-curves} part (2), one has
\[
        H^2_{\mathrm{\acute et}}
        (\mathscr U_{\bar k},\operatorname{Ad}\bar r)=0
\]
and it follows that
\[
        H^2(G_\infty,\operatorname{Ad}\bar r)=0.
\]
Thus the deformation functors are unobstructed and the associated deformation rings are formally smooth. 

\par It follows from Theorem \ref{thm:stacks-torsion-curves} part (4) that $H^1(G_\infty,\operatorname{Ad}\bar r)$ is finite-dimensional over \(\mathbb F\). The space of $1$-coboundaries is clearly finite dimensional and hence $Z^1(G_\infty,\operatorname{Ad}\bar r)$ is finite-dimensional. Therefore, the deformation rings are formal power series rings over $\cO$ in finitely many variables.
\end{proof}

\bibliographystyle{alpha}
\bibliography{references}
\end{document}